\documentclass[12pt, twoside]{article}
\usepackage{amsmath,amsthm,amssymb}
\usepackage{times}
\usepackage{enumerate}
\usepackage{bm}
\usepackage[all]{xy}
\usepackage{mathrsfs}
\usepackage{amscd}
\usepackage{color}
\def\titlerunning#1{\gdef\titrun{#1}}
\makeatletter
\def\author#1{\gdef\autrun{\def\and{\unskip, }#1}\gdef\@author{#1}}
\def\address#1{{\def\and{\\\hspace*{18pt}}\renewcommand{\thefootnote}{}%
		\footnote {#1}}%
	\markboth{\autrun}{\titrun}}
\makeatother
\def\email#1{e-mail: #1}

\def\keywords#1{\par\medskip
	\noindent\textbf{Keywords.} #1}

\newtheorem{theorem}{Theorem}[section]
\newtheorem{corollary}[theorem]{Corollary}
\newtheorem{lemma}[theorem]{Lemma}
\newtheorem{proposition}[theorem]{Proposition}
\theoremstyle{definition}
\newtheorem{definition}[theorem]{Definition}
\newtheorem{remark}[theorem]{Remark}

\theoremstyle{example}

\numberwithin{equation}{section}
\newtheorem{lemma A.}{Lemma  A.}

\xyoption{all}

\def \C {\mathbb{C}}

\def \a {\alpha }
\def \b {\beta}

\def \de {\delta}
\def \De {\Delta}

\def \La {\Lambda}
\def\w {\omega}
\def\Om{\Omega}

\def\na {\nabla}

\begin{document}
\baselineskip=17pt

\titlerunning{Vanishing theorems on complete Riemannian manifold with a parallel $1$-form}
\title{Vanishing theorems on complete Riemannian manifold with a parallel $1$-form }

\author{Teng Huang and Qiang Tan}

\date{}

\maketitle

\address{T. Huang: School of Mathematical Sciences, University of Science and Technology of China; Key Laboratory of Wu Wen-Tsun Mathematics, Chinese Academy of Sciences, Hefei, 230026, P.R. China; \email{htmath@ustc.edu.cn;htustc@gmail.com}}
\address{Q. Tan: Faculty of Science, Jiangsu University, Zhenjiang, Jiangsu 212013, P.R. China; \email{tanqiang@ujs.edu.cn}}
\begin{abstract}
In this article, we first consider the $L^{2}$ \textit{Morse-Novikov cohomology} on a complete Riemannian manifold $M$ equipped with a parallel $1$-form which includes Vaisman manifold. Based on a vanishing theorem of $L^{2}$ \textit{Morse-Novikov cohomology}, we prove that the $L^{2}$-harmonic forms on $M$ are identically zero. 
\end{abstract}
\section{Introduction}
Let $M$ be a complete $n$-dimensional Riemannian manifold. A basic question, pertaining both the function theory and topology on $M$, is: when are there non-trivial harmonic $k$-forms on $M$? When $X$ is not compact, a growth condition on the harmonic forms at infinity must be imposed, in order that the answer to this question would be useful. A natural growth condition is square-integrable, if $\Om^{k}_{(2)}(X)$ denotes the $L^{2}$-forms of degree $k$ on $M$ and $\mathcal{H}^{k}_{(2)}(X)$ the harmonic forms in $\Om^{k}_{(2)}(X)$. One version of this basic question is: what is the structure of $\mathcal{H}^{k}_{(2)}(X)$?

The Hodge theorem for compact manifolds states that every real cohomology class of a compact manifold $M$ is represented by a unique harmonic form. That is, the space of solutions to the differential equation $(d+d^{\ast})\a=0$ on $L^{2}$-forms over $M$ is a space that depends on the metric of $M$. This space is canonically isomorphic to the purely topological real cohomology space of $M$. The study of $\mathcal{H}^{k}_{(2)}(M)$, a question of so-called $L^{2}$-cohomology of $M$, is rooted in the  attemption extending  Hodge theory to non-compact manifolds. No such result holds in general for complete non-compact manifolds, but there are numerous partial results about the $L^{2}$-cohomology of non-compact manifold. The study of the $L^{2}$-harmonic forms on a complete Riemannian manifold is a very  fascinating and significant subject. There has been some recent interest in the study of $L^{2}$ harmonic forms on certain non-compact moduli spaces occurring in gauge theory \cite{Hitchin}.

Let $M$ be a Riemannian manifold equipped with a differential form $\w$. This form is called parallel if $\w$ is preserved by the Levi-Civita connection: $\na\w=0$. This identity gives a powerful restriction on the holonomy group $Hol(M)$. The structure of $Hol(M)$ and its relation to the geometry of a manifold is one of the main subjects of Riemannian geometry of the last 50 years.  Suppose that $(M,\w)$ is a complete, K\"{a}hler manifold of complex dimension $n$ and $\w$ is the K\"{a}hler form on $M$. In K\"{a}hler geometry (holonomy $U(n)$) the parallel forms are the K\"{a}hler form and its powers. The metric induced by K\"{a}hler form allows one to define the class of square-integrable forms of all bi-degrees, $\Om^{p,q}_{(2)}(M)$.  We denote by $\mathcal{H}_{(2)}^{p,q}(M)$ the space of $L^{2}$-harmonic $(p,q)$-forms. There are many articles study the K\"{a}hlerian case \cite{CX,Don1, Don2, DF, Gromov,McNeal}. In many situations, e.g., $(M,\w)=hyperbolic$ upper half plane in $\C^{n}$, it happens that $\mathcal{H}^{p,q}_{(2)}(M)=\{0\}$ unless $p+q=n$. The middle dimension, when $p+q=n$, is always a special case. For example, there are no results in \cite{McNeal} about $L^{2}$ harmonic forms in these dimensions.

The main object of the present paper is the following notion. Let $(M,g,\eta)$ be a simply-connected manifold equipped with a parallel $1$-form $\eta$. We prove a vanishing theorem as follows.
\begin{theorem}\label{T3}
Let $M$ be a complete, simply-connected $n$-manifold and $\eta$ a closed non-zero $1$-form on $M$. Suppose that $g$ is a Riemannian metric on $M$ such that $\eta$ is parallel with respect to $g$: $\na \eta=0$. Then for any $0\leq k\leq n$, the spaces $\mathcal{H}^{k}_{(2)}(M)$ of $L^{2}$-harmonic $k$-forms are trivial.
\end{theorem}
Let $(M,J,g)$ be a connected complex Hermitian manifold of complex dimension
at least $2$. Denote by $\w$ its fundamental Hermitian two-form, with the
convention $\w(X,Y)=g(X,JY)$.  A locally conformally K\"{a}hler manifold $(M,J,g,\w,\theta)$ is called Vaisman if $\na\theta=0$, where $\na$ is the Levi-Civita connection of metric $g$. The Vaisman manifold is a distinguished class among the LCK manifolds. 
\begin{theorem}\label{T2}
Let $(M,J,g,\w,\theta)$ be a complete, simply-connected, Vaisman manifold of complex dimension $n$.  Then for any $0\leq k\leq n$, the spaces $\mathcal{H}^{k}_{(2)}(M)$ of $L^{2}$-harmonic $k$-forms are trivial.
\end{theorem}
The organization of this paper is as follows. In section 2, we recall some basic notations and well-known results on $L^{2}$-harmonic forms. Using the ideas in \cite{LLMP}, we prove that on a complete manifold, the space of the $L^{2}$ \textit{Morse-Novikov} harmonics with respect to a non-zero  parallel $1$-form is trivial (see Theorem \ref{T1}). In Section 3, we extend the ideas in \cite{CX} to the Riemannian manifold with parallel $1$-form case. Then we can prove that the $L^{2}$-harmonic form is also in the space of $L^{2}$ \textit{Morse-Novikov} harmonic forms (see Lemma \ref{P9}). At last, using these results, we deduce that on a complete simply-connected Vaisman manifold, the space of $L^{2}$-harmonic forms also is trivial.
\section{Preliminaries}
\subsection{$L^{2}$-harmonic forms}
Let $(M,g)$ be a complete Riemannian manifold. Let $\Om^{k}(M)$ and $\Om^{k}_{0}(M)$ denote the smooth $k$-forms on $M$ and the smooth $k$-forms with compact support on $M$. Let $\langle\cdot,\cdot\rangle$ denote the pointwise inner product on $\Om^{\ast}(M)$ given by $g$. The global inner product is defined by
$$(\a,\b)=\int_{M}\langle\a,\b\rangle d\rm{Vol}.$$
We also write $|\a|^{2}=\langle\a,\a\rangle$, $\|\a\|^2_{L^{2}(M,g)}=\int_{M}|\a|^{2}d\rm{Vol}$ and let
$$\Om^{k}_{(2)}(M)=\{\a\in\Om^{k}(M):\|\a\|_{L^{2}(M,g)}<\infty\}.$$
The operator of exterior differentiation is $d:\Om^{k}_{0}(M)\rightarrow\Om^{k+1}_{0}(M)$, and it satisfies $d^{2}=0$; its formal adjoint is $\de:\Om^{k+1}_{0}(M)\rightarrow\Om^{k}_{0}(M)$; we have
$$\forall\a\in\Om^{k}_{0}(M),\ \forall\b\in\Om^{k+1}_{0}(M),\ \int_{M}\langle d\a,\b\rangle=\int_{M}\langle\a,\de\b\rangle.$$
We can define
\begin{equation*}
\begin{split}
\mathcal{H}_{(2)}^{k}(M)=\{\a\in \Om^{k}_{(2)}(M): d\a=0,\ \de\a=0 \}.\\
\end{split}
\end{equation*}
Noticing that the operator $d+\de$ is elliptic. By elliptic regularity, we then have : $\mathcal{H}^{k}_{(2)}(M)\subset\Om^{k}(M)$. The space $\Om^{k}_{(2)}(M)$ has the following of Hodge-de Rham-Kodaira orthogonal decomposition
$$\Om^{k}_{(2)}(M)=\mathcal{H}^{k}_{(2)}(M)\oplus\overline{d(\Om^{k-1}_{0}(M))}\oplus\overline{\de(\Om^{k+1}_{0}(M))  },$$
where the closure is taken with respect to the $L^{2}$ topology \cite{Carron}.

\subsection{Morse-Novikov cohomology}
Let $M$ be a differential manifold and $\eta$ a closed $1$-form on $M$. The \textit{Morse-Novikov cohomology} of a manifold $M$ refers to the cohomology of complex of smooth real form $\Om^{\ast}(M)$, with the differential operator  defined as follow
\begin{equation}\label{E2}
d_{\eta}=d+e(\eta),
\end{equation}
$d$ being the exterior differential and $e(\eta)$ the operator given by
\begin{equation}\label{E3}
e(\eta)(\a)=\eta\wedge\a,\ \forall\a\in\Om^{\ast}(M).
\end{equation}
Recall that the Morse-Novikov cohomology, also known as Lichnerowicz cohomology \cite{Lic} which defined independently by Novikov \cite{Nov} and Guedira-Lichnerowicz \cite{GL}. That is a cohomology of the complex $(\Om^{\ast}(M),d_{\eta})$
\begin{equation}\label{E1}
\Om^{0}(M)\xrightarrow{d_{\eta}}\Om^{1}(M)\xrightarrow{d_{\eta}}\Om^{2}(M)\xrightarrow{}\cdots
\end{equation}
Denote by $H_{\eta}^{\ast}(M)$ the cohomology of the complex $(\Om^{\ast}(M),d_{\eta})$. In fact, the sequence above is an acyclic resolution for $\ker d_{\eta}$, as each  $\Om^{\ast}(M)$ is soft, \cite[Proposition 2.1.6 and Theorem 2.1.9]{Dimca}. Thus, by taking global sections in (\ref{E1}), we compute the cohomology groups of $M$ with values in the sheaf $\ker d_{\eta}$, $H^{i}(M,\ker d_{\eta})$. What we obtain is actually the Morse-Novikov
cohomology.
\begin{proposition}(\cite[Proposition 4.4]{LLMP})
Let $M$ be a differentiable manifold and $\eta$ a closed $1$-form on $M$. Then, \\
(i) The differential complex $(\Om^{\ast}(M),d_{\eta})$ is elliptic. Thus, if $M$ is compact, the cohomology groups $H_{\eta}^{k}$ have finite dimension.\\
(ii) If $\eta$ is exact, then $H_{\eta}^{k}(M)\cong H^{k}_{dR}(M)$.
\end{proposition}
In general, if the $1$-form $\eta$ is not exact then $H_{\eta}^{k}(M)\ncong H^{k}_{dR}(M)$. We recall some results proved by Gu\'{e}dira-Lichnerowicz \cite{GL} which will be useful in the sequel. Let $M$ be a smooth $n$-dimensional manifold, $\eta$ a closed $1$-form on $M$ and  $g$ be a Riemannian metric of $M$. Consider the vector field $U$ on $M$ characterized by the condition $\eta(X)=g(X,U)$, for all vector filed $X$ on $M$. Denote by $\textit{i}_{U}$ the contraction by the vector field $U$, that is, for any  $\a\in\Om^{k}(M)$, we have
\begin{equation}\label{E4}
\textit{i}_{U}(\a)=(-1)^{nk+n}(\ast\circ e(\eta)\circ\ast)(\a).
\end{equation}
Denote by $\ast$ the Hodge star operator with respect to metric $g$. Then, we define the operator $\de_{\eta}:\Om^{k}_{0}(M)\rightarrow\Om^{k-1}_{0}(M)$ by
\begin{equation}\label{E9}
\de_{\eta}=\de+\textit{i}_{U}.
\end{equation}
Then, it easy to prove that 
$$(d_{\eta}\a,\b) =(\a,\de_{\eta}\b),\  \forall \a\in\Om_{0}^{k-1}(M), \b\in\Om_{0}^{k}(M).$$
If $M$ is compact, since the complex $(\Om^{\ast}(M),d_{\eta})$ is elliptic, we obtain an orthogonal decomposition of the space $\Om^{k}(M)$ as follows
\begin{equation}\label{E5}
\Om^{k}(M)=\mathcal{H}^{k}_{\eta}(M)\oplus d_{\eta}(\Om^{k-1}(M))\oplus\de_{\eta}(\Om^{k+1}(M)),
\end{equation}
where  $$\mathcal{H}_{\eta}^{k}(M)=\{\a\in\Om^{k}(M):d_{\eta}(\a)=0, \de_{\eta}(\a)=0 \}.$$
From (\ref{E5}), it follows that $H_{\eta}^{k}(M)\cong\mathcal{H}^{k}_{\eta}(M)$.  But Morse-Novikov cohomology $H^{i}_{\eta}(M)$ is different to de Rham cohomology, it is not a topological invariant, it depends on $[\eta]\in H^{1}_{dR}(M)$.  Also, Riemannian properties involving this one-form can be important. For instance, it was shown in \cite{LLMP} that if there exists a Riemannian metric $g$ and a closed one-form $\eta$ on a compact manifold $M$  such that $\eta$ is parallel with respect to $g$, then  $H^{i}_{\eta}(M)=0$ for any $i\geq 0$.

In complete non-compact case, we now define the spaces of generalized $L^{2}$-harmonic forms  as follows:
$$\mathcal{H}_{(2),\eta}^{k}(M)=\{\a\in\Om^{k}_{(2)}(M):d_{\eta}(\a)=0, \de_{\eta}(\a)=0\}.$$
Following the idea in \cite{LLMP}, we have
\begin{theorem}\label{T1}
Let $M$ be a complete $n$-dimensional manifold and $\eta$ a closed non-zero $1$-form on $M$. Suppose that $g$ is a Riemannian metric on $M$ such that $\eta$ is parallel with respect to $g$: $\na\eta=0$. Then, $\mathcal{H}^{\ast}_{(2),\eta}$ is trivial.
\end{theorem}
\begin{proof}
Since $\eta$ is a parallel and non-null it follows that $|\eta|=c$, with $c$ constant, $c>0$. Assume, without the loss of generality, that $c=1$. Note that if $c\neq 1$, we can consider the Riemannian metric $g'=c^{2}g$ and it is clear that the module of $\eta$ with respect to $g'$ is $1$ and that $\eta$ is also parallel with respect to $g'$. Under the hypothesis  $c=1$, we have that
\begin{equation}\label{E6}
\eta(U)=1.
\end{equation}
Using that $\eta$ is parallel and that $U$ is Killing, we obtain that (see (\ref{E4}))
\begin{equation}\label{E7}
\mathcal{L}_{U}=-\de\circ e(\eta)-e(\eta)\circ\de,
\end{equation}
\begin{equation}\label{E8}
\de\circ\mathcal{L}_{U}=\mathcal{L}_{U}\circ\de.
\end{equation}
From (\ref{E2}--\ref{E9}), (\ref{E6}) and (\ref{E8}), we deduce the following relations:
\begin{equation}\label{E10}
d_{\eta}\circ\textit{i}_{U}=-\textit{i}_{U}\circ d_{\eta}+\mathcal{L}_{U}+Id,\ \de_{\eta}\circ\textit{i}_{U}=-\textit{i}_{U}\circ\de_{\eta},
\end{equation}
\begin{equation}\label{E11}
d_{\eta}\circ\mathcal{L}_{U}=\mathcal{L}_{U}\circ d_{\eta}, \de_{\eta}\circ\mathcal{L}_{U}=\mathcal{L}_{U}\circ\de_{\eta},
\end{equation}
where $Id$ denotes the identity transformation.

Let $\xi:\mathbb{R}\rightarrow\mathbb{R}$ be smooth, $0\leq\xi\leq1$,
\begin{equation*}
\xi(t)=\left\{
\begin{aligned}
1, &  & t\leq0 \\
0,  &  & t\geq1
\end{aligned}
\right.
\end{equation*}
and consider the compactly supported function
\begin{equation*}
f_{j}(x)=\xi(\rho(x_{0},x)-j),
\end{equation*}
where $j$ is a positive integer and $\rho(x_{0},x)$ stands for the Riemannian distance between $x$ and a base point $x_{0}$.
On the other hand, (\ref{E7}) implies that
\begin{equation*}
\begin{split}
\langle \mathcal{L}_{U}\a,f_{j}\a\rangle&=-\langle\a, d\textit{i}_{U}(f_{j}\a)+\textit{i}_{U}d(f_{j}\a)\rangle\\
&=-\langle\a, f_{j}\mathcal{L}_{U}\a\rangle-\langle\a, \textit{i}_{U}(df_{j})\wedge\a\rangle\\
\end{split}
\end{equation*}
for all $\a\in\Om^{k}(M)$. Here we use the identity
\begin{equation*}
\mathcal{L}_{U}(\a\wedge\b)=(\mathcal{L}_{U}\a)\wedge\b+\a\wedge(\mathcal{L}_{U}\b).
\end{equation*}
Noting that
\begin{equation*}
|\langle\a, \textit{i}_{U}(df_{j})\wedge\a\rangle|\leq |df_{j}|\cdot|\a|^{2}.
\end{equation*}
Thus, there exists a subsequence $j_{i}\rightarrow\infty$ as $i\rightarrow\infty$ such that
\begin{equation}\label{E12}
\lim_{i\rightarrow\infty}\langle f_{j_{i}}\a, \mathcal{L}_{U}\a\rangle=0.
\end{equation}
Now, if $\a\in\mathcal{H}_{(2),\eta}^{k}$ then, using (\ref{E10}), we have that
\begin{equation}\label{E13}
\mathcal{L}_{U}\a=-\a+d_{\eta}(\textit{i}_{U}\a).
\end{equation}
Taking the $L^{2}$-inner of (\ref{E3}) with $f_{j_{i}}\a$ and integrating by parts, we obtain
\begin{equation}\label{E00}
\begin{split}
(\mathcal{L}_{U}\a,f_{j_{i}}\a)&=-(\a,f_{j_{i}}\a)+(\textit{i}_{U}\a,\de_{\eta}(f_{j_{i}}\a) )\\
&=-(\a,f_{j_{i}}\a)+(\textit{i}_{U}\a,(\pm)\ast(df_{j_{i}}\wedge\ast\a)+f_{j_{i}}\de_{\eta}\a )\\
&=-(\a,f_{j_{i}}\a)+(\textit{i}_{U}\a,(\pm)\ast(df_{j_{i}}\wedge\ast\a) ).\\
\end{split}
\end{equation}
Since $0\leq f_{j_{i}}\leq 1$ and $\lim_{i\rightarrow\infty}f_{j_{i}}(x)\a(x)=\a(x)$, it follows from the dominated convergence theorem that
\begin{equation}\label{E01}
 \lim_{i\rightarrow\infty}(\a,f_{j_{i}}\a)=\|\a\|^{2}.
 \end{equation}
 Since $|U|=1$ and $supp(df_{j_{i}})\subset B_{j_{i}+1}\backslash B_{j_{i}}$, one obtains
\begin{equation}\label{E02}
|(\textit{i}_{U}\a,\ast(df_{j_{i}}\wedge\ast\a) )|\leq C\int_{B_{j_{i}+1}\backslash B_{j_{i}}}|\a(x)|^{2}\rightarrow 0,\ as\ i\rightarrow\infty,
\end{equation}
where $C$ is a constant independent of $j_{i}$. It now follows from (\ref{E12}) and (\ref{E00})--(\ref{E02}) that $\a=0$. This proves that $\mathcal{H}^{k}_{(2),\eta}(M)=\{0\}$.
\end{proof}
\section{Vanishing theorems}
\subsection{Riemannian manifold with parallel $1$-form}
In this section, we recall some notations and definitions on differential geometry \cite{Ver}. Let $M$ be a smooth Riemannian manifold. We denote by $\Om^{\ast}(X)$ the smooth forms on $M$. Given an odd or even from $\a\in\Om^{\ast}(M)$, we denote by $\tilde{\a}$ its parity, which is equal to $0$ for even forms, and $1$ for odd forms. An operator $f\in\rm{End}(\La^{\ast}(M))$ preserving parity is called $even$, and one exchanging odd and even forms is odd; $\tilde{f}$ is equal to $0$ for even forms and $1$ for odd ones. Given a $C^{\infty}$-linear map $\Om^{1}(M)\xrightarrow{p}\Om^{odd}(M)$ or  $\Om^{1}(M)\xrightarrow{p}\Om^{even}(M)$, $p$ can be uniquely extended to a $C^{\infty}$-linear derivation $\rho$ on $\Om^{\ast}(M)$, using the rule
\begin{equation*}
\begin{split}
&\rho|_{\Om^{0}(M)}=0,\\
& \rho|_{\Om^{1}(M)}=p,\\
& \rho(\a\wedge\b)=\rho(\a)\wedge\b+(-1)^{\tilde{\rho}\tilde{\a}}\a\wedge\rho(\b).\\
\end{split}
\end{equation*}
Then, $\rho$ is an even (or odd) differentiation of the graded commutative algebra $\Om^{\ast}(M)$. Verbitsky gave a definition of the structure operator of $(M,\w)$, see \cite{Ver} Definition 2.1.
\begin{definition}\label{D2.1}
Let $M$ be a Riemannian manifold equipped with a parallel differential $k$-form $\eta$. Consider an operator $\underline{C}:\Om^{1}(M)\rightarrow\Om^{k-1}(M)$ mapping $\a\in\Om^{1}(M)$ to $\ast(\ast\eta\wedge\a)$. The corresponding differentiation
$$C:\Om^{\ast}(M)\rightarrow\Om^{\ast+k-2}(M)$$
is called the structure operator of $(M,\eta)$.
\end{definition}
\begin{definition}(\cite[Definition 2.3]{Ver})
Let $M$ be a Riemannian manifold, and $\eta\in\Om^{k}(M)$  a differential form, which is parallel with respect to the Levi-Civita connection. Denote by $d_{C}$ the supercommutator
$$\{d, C\}:= dC-(-1)^{\tilde{C}}Cd.$$
This operator is called the twisted de Rham operator of $(M,\w)$. Being a graded commutator of two graded differentiations, $d_{C}$ is also a graded differentiation of $\Om^{\ast}(M)$.
\end{definition}
\begin{lemma}\label{L4}(\cite[Proposition 2.5]{Ver})
Let $M$ be a Riemannian manifold equipped with a parallel differential $k$-form $\eta$, and  $L_{\eta}$ the operator $\a\mapsto\a\wedge\eta$. Then
$$d_{C}=\{L_{\eta},d^{\ast}\},$$
where $\{\cdot,\cdot\}$ denotes the supercommutator, and $d^{\ast}$ is the adjoint to $d$.
\end{lemma}
\begin{remark}
If $\eta$ is a parallel one form on $M$, then in fact the structure operator of $M$ is $C=\textit{i}_{\eta}$. The operator $d_{C}$ is the Lie derivative $\mathcal{L}_{\eta}$ since $\{d,\textit{i}_{\eta}\}=\mathcal{L}_{\eta}$.
\end{remark}
We recall some results which proved by Verbitsky (See \cite[Proposition 2.5 and Corollary 2.9]{Ver}  ).
\begin{proposition}\label{P5}
Let $M$ be a Riemannian manifold equipped with a parallel differential $k$-form $\eta$, $d_{C}$ the twisted de Rham operator constructed above and $d^{\ast}_{C}$ its Hermitian adjoint. Then,\\
(i) The following supercommutators vanish:
$$\{d,d_{C}\}=0,\ \{d,d_{C}^{\ast}\}=0,\ \{d^{\ast},d_{C}\}=0,\ \{d^{\ast},d_{C}^{\ast}\}=0.$$
(ii) The Laplacian $\De=\{d,d^{\ast}\}$ commutes with $L_{\eta}:\a\mapsto\a\wedge\eta$ and its adjoint operator $\La_{\eta}$ which is denoted as $\La_{\eta}:\Om^{i}(M)\rightarrow\Om^{i-k}(M)$.
\end{proposition}
Following Proposition \ref{P5}, if $\a$ is a harmonic form on $M$, then $\a\wedge\eta$ is harmonic.
\begin{corollary}\label{C1}
Let $M$ be a complete manifold and $\eta$ a closed non-zero $1$-form on $M$. Suppose that $g$ is a Riemannian metric on $M$ such that $\eta$ is parallel with respect to $g$: $\na\eta=0$. If $\a$ is a $L^{2}$-harmonic $k$-form on $M$, then $\eta\wedge\a$ is a $L^{2}$-harmonic $(k+1)$-from.
\end{corollary}
\begin{proof}
Since the $1$-form $\eta$ is a parallel, $|\eta|=const.$ Therefore, $\eta\wedge\a\in\Om^{k+1}_{(2)}(M)$. Following Proposition \ref{P5}, it implies that $$\De(\eta\wedge\a)=0.$$ 
Then, we have $d(\eta\wedge\a)=0$ and $d^{\ast}(\eta\wedge\a)=0$.
\end{proof}
Let $(M,g)$ be a complete Riemannian manifold. A differential form $\a$ is called $d$(bounded) if there exists a form $\b$ on $M$ such that $\a=d\b$ and
$$\|\b\|_{L^{\infty}(M,g)}=\sup_{x\in M}|\b(x)|_{g}<\infty.$$
It is obvious that if $M$ is compact, then every exact form is $d$(bounded). However, when $M$ is not compact, there exist smooth differential forms which are exact but not $d$(bounded). For instance, on $\mathbb{R}^{n}$, $\a=dx^{1}\wedge\cdots\wedge dx^{n}$ is exact, but it is not $d$(bounded).

Let us recall some concepts introduced by Cao-Xavier in \cite{CX}. A differential form $\a$ on a complete non-compact Riemannian manifold $(M,g)$ is called $d$(sublinear) if there exist a differential form $\b$ and a number $c>0$ such that
$\a=d\b$
and
\begin{equation*}
\begin{split}
&|\a(x)|_{g}\leq c,\\
& |\b(x)|_{g}\leq c(1+\rho(x,x_{0})),\\
\end{split}
\end{equation*}
where $\rho(x,x_{0})$ stands for the Riemannian distance between $x$ and a base point $x_{0}$ with respect to $g$.

In \cite{Huang}, the author extended the idea of Cao-Xavier's \cite{CX} to the case of Riemannian manifold equipped with a parallel differential form. We then have a result as follows. Here, we give a proof in detail for the reader’s convenience.
\begin{theorem}\label{T9}\cite[Theorem 2.9]{Huang}
Let $(M,\eta)$ be a Riemannian manifold equipped with a parallel differential $k$-form $\eta$. If $\eta=d\b$ is $d$(sublinear), then for any $\a\in\mathcal{H}^{p}_{(2)}(X)$, we have
$$\eta\wedge\a=0.$$
\end{theorem}
\begin{proof}
Let $\{f_{j}\}_{j=0,1,2,\cdots}$ be the compactly supported functions which are the same as the functions in the proof of Theorem \ref{T1}. Let $\a$ be a harmonic $p$-form in $L^{2}$, and consider the form $\nu=\b\wedge\a$. Observing that $d^{\ast}(\eta\wedge\a)=0$ since $\eta\wedge\a\in\mathcal{H}_{(2)}^{p+k}(X)$  and noticing that $f_{j}\nu$ has compact support, one has
\begin{equation}\label{E23}
\begin{split}
0&=\big{(}d^{\ast}(\eta\wedge\a),f_{j}\nu\big{)}\\
&=\big{(}\eta\wedge\a,d(f_{j}\nu)\big{)}\\
&=\big{(}\eta\wedge\a, f_{j}\eta\wedge\a\big{)}+\big{(}\eta\wedge\a, df_{j}\wedge\b\wedge\a\big{)}.\\
\end{split}
\end{equation}
Since $0\leq f_{j}\leq 1$ and $\lim_{j\rightarrow\infty}f_{j}(x)(\eta\wedge\a)(x)=(\eta\wedge\a)(x)$, it follows from the dominated convergence theorem that
\begin{equation}\label{E24}
	\lim_{j\rightarrow\infty}\big{(}\eta\wedge\a, f_{j}\eta\wedge\a\big{)}=\|\eta\wedge\a\|^{2}.
\end{equation}
	Since $\eta$ is bounded, $supp(df_{j})\subset B_{j+1}\backslash B_{j}$ and $|\b(x)|=O(\rho(x_{0},x))$, one obtains
	\begin{equation}\label{E22}
|\big{(}\eta\wedge\a,df_{j}\wedge\b\wedge\a\big{)}|\leq (j+1)C\int_{B_{j+1}\backslash B_{j}}|\a(x)|^{2}d{\rm{Vol}},
	\end{equation}
	where $C$ is a constant independent of $j$.
	
	We claim that there exists a subsequence $\{j_{i}\}_{i\geq1}$ such that
	\begin{equation}\label{E21}
	\lim_{i\rightarrow\infty}(j_{i}+1)\int_{B_{j_{i}+1}\backslash B_{j_{i}}}|\a(x)|^{2}d{\rm{Vol}}=0.
	\end{equation}
	If not, there would exist a positive constant $a$ such that
	$$\lim_{i\rightarrow\infty}(j_{i}+1)\int_{B_{j_{i}+1}\backslash B_{j_{i}}}|\a(x)|^{2}{\rm{Vol}}\geq a>0,\ j\geq1.$$
	This inequality implies
	\begin{equation}\nonumber
	\begin{split}
	\int_{M}|\a(x)|^{2}{\rm{Vol}}&=\sum_{j=0}^{\infty}\int_{B_{j+1}\backslash B_{j}}|\a(x)|^{2}{\rm{Vol}}\\
	&\geq a\sum_{j=0}^{\infty}\frac{1}{j+1}\\
	&=+\infty\\
	\end{split}
	\end{equation}
 a contradiction to the assumption $\int_{M}|\a(x)|^{2}\rm{Vol}<\infty$. Hence, there exists a subsequence $\{j_{i}\}_{i\geq1}$ for which (\ref{E21}) holds. Using (\ref{E22}) and (\ref{E21}), one obtains
	\begin{equation}\label{E25}
	\lim_{i\rightarrow\infty}\big{(}\eta\wedge\a, df_{j}\wedge\b\wedge\a\big{)}=0
	\end{equation}
	It now follows from (\ref{E23}), (\ref{E24}) and (\ref{E25}) that $\eta\wedge\a=0$.
\end{proof}
There is a very-known result. Let $M$ be a compact Riemannian manifold, $\a$ be a closed $1$-from and $\pi:\tilde{M}\rightarrow M$ be the universal covering. Then the pull back form $\pi^{\ast}(\a)$ is $d$(sublinear) (see \cite[Proposition 1]{JZ}). We observe a useful lemma as follows.
\begin{lemma}\label{L1}
Let $M$ be a complete, non-compact Riemannian manifold. If the $C^{\infty}$-function $f$ on $M$ satisfies $\na^{2}f=0$, then for any $x\in M$,
$$|f(x)|\leq c(\rho(x,x_{0})+1),$$
where $c$ is a uniform positive constant.	
\end{lemma}
\begin{proof}
Sine $\na^{2}f=0$, $|\na f|=const.$ Let $x_{0}$ be a fix point on $M$. For any point $x$ in $M$, there exists a geodesic  $s:[0,1]\rightarrow M$ such that $s(0)=x_{0}$ and $s(1)=x$. Thus
$$|f(s(1))-f(s(0))|\leq c\rho|\na f|,$$
where $\rho$ is the Riemannian distance between $x_{0}$ and $x$, $c$ is a positive constant independent  on $x\in M$.
\end{proof}
We then have
\begin{lemma}\label{P9}
Let $M$ be a complete, simply-connected manifold and $\eta$ a closed non-zero $1$-form on $M$. Suppose that $g$ is a Riemannian metric on $M$ such that $\eta$ is parallel with respect to $g$: $\na\eta=0$. Then for any $L^{2}$-harmonic $k$-form $\a$ on $M$, we have
\begin{equation}
e(\eta)(\a)=0,\ \textit{i}_{U}(\a)=0.
\end{equation}
In particular, for any $k\geq0$, $$\mathcal{H}^{k}_{(2)}(M)\subset\mathcal{H}^{k}_{(2),\eta}(M).$$
\end{lemma}
\begin{proof}
Since $M$ is simply-connected and $\eta$ is a closed $1$-form on $M$, there is a function $f$ on $M$ such that $\eta=df$. Therefore, $\na^{2}f=0$ since $\na\eta=0$. By the Lemma \ref{L1}, $\eta$ is $d$(sublinear). Let $\a$ be a $L^{2}$-harmonic $k$-form. Then, $\ast\a$ is also a $L^{2}$-harmonic $(n-k)$-form. Following Theorem \ref{T9}, it implies that
\begin{equation*}
\begin{split}
&e(\eta)(\a)=\eta\wedge\a=0,\\
&\textit{i}_{U}(\a)=(-1)^{nk+n}\ast e(\eta)(\ast\a)=0.\\
\end{split}
\end{equation*}
Therefore, 
$$d_{\eta}\a=d\a+e(\eta)\a=0$$ 
and 
$$\de_{\eta}\a=\de\a+\textit{i}_{U}\a=0,$$ 
i.e., $\a\in \mathcal{H}^{k}_{(2),\eta}(M)$.
\end{proof}
\begin{proof}[\textbf{Proof of Theorem \ref{T3}}.]
The conclusion follows from Theorem \ref{T1} and Lemma \ref{P9}.
\end{proof}
\subsection{Vaisman manifolds}
In this section we will give the necessary definitions and properties of locally conformally K\"{a}hler (LCK) manifolds.  Let $M$ be a connected, smooth manifold of complex dimension $n$ and $J$ be an integrable complex structure of $M$. For a Hermitian metric $g$ on $M$, we denote by $\na$ the Levi-Civita connection and by $\w$ the fundamental two-form defined as $\w(X,Y)=g(JX,Y)$.

A locally conformally K\"{a}hler manifold is a complex manifold $X$ covered by a system of open subsets $U_{\a}$ endowed with local K\"{a}hler metrics $g_{\a}$,  conformal on overlaps $U_{\a}\cap U_{\b}$: $g_{\a}=c_{\a\b}g_{\b}$. The metrics $e^{-f_{\a}}g_{\a}$ glue to a global metric whose associated $2$-form $\w$ satisfies the integrability condition $d\w=\theta\wedge\w$, thus being locally conformal with the K\"{a}hler metrics $g_{\a}$. Here $\theta|_{U_{\a}}=df_{\a}$. The closed $1$-form $\theta$ is called the Lee form. This gives another definition of an LCK structure, which will be used in this paper.
\begin{definition}
	Let $(M,g,\w)$ be a complex Hermitian manifold, $\dim_{\C}M>1$, with $$d\w=\theta\wedge\w,$$ 
	where $\theta$ is a closed 1-form. Then $M$ is called a locally conformally K\"{a}hler (LCK) manifold.
\end{definition}
If one performs a conformal change, $\w_{1}=e^{f}\w$, the Lee form  $\theta$ changes to $\theta_{1}=\theta+df$. The cohomology class $[\theta]\in H^{1}_{dR}(M)$ is an important invariant of an LCK-manifold. Clearly, $[\theta]\in H^{1}_{dR}(M)$ vanishes if and only if $\w$ is conformally equivalent to a K\"{a}hler structure. In this case $(M,\w)$ is called globally conformally K\"{a}hler.

A compact LCK manifold never admits a K\"{a}hler structure, unless the cohomology class $[\theta]\in H_{dR}^{1}(M)$ vanishes  \cite{Vaisman1980}. In many situations, the LCK structure becomes useful for the study of topology and complex geometry of a non-K\"{a}hler manifold.  An LCK-form $\w$ on an LCK-manifold satisfies $d\w=\w\wedge\theta$, therefore it is $d_{-\theta}$-closed. The cohomology class $[\w]\in H^{2}_{\theta}(M)$ is called the Morse–Novikov class of the LCK-manifold. It is an invariant of the LCK-manifold, roughly analogous to the K\"{a}hler class on a K\"{a}hler manifold. In \cite{OV2009}, the author defined three cohomology invariants, the Lee class, the Morse-Novikov class, and the Bott-Chern class, of an LCK-structure. These invariants play together the same role as the K\"{a}hler class in K\"{a}hler geometry  \cite{OV2010,OV2012,OV2018}.  Among the LCK manifolds, a distinguished class is the following:
\begin{definition}
	Let $(M,g,\w,\theta)$ be an LCK manifold and $\na$ its Levi-Civita connection.  We say that $M$ is an LCK manifold with parallel Lee form $\theta$, or Vaisman manifold, if $\na\theta=0$. 
\end{definition}
Let $(M,J,\theta)$ be a Vaisman manifold. Since the Lee form $\theta$ is parallel, $|\theta|=2c$ for some $c\in\mathbb{R}$, $c\neq 0$. We adopt the notations
$$u=|\theta|^{-1}\theta,\ U=u^{\sharp},\ v=-u\circ J,\ V=-JU.$$
We recall that given a real $(2n-1)$-dimensional $C^{\infty}$ differentiable manifold $N$ and $c\in\mathbb{R}$, $c\neq 0$, a $c$-Sasakian structure on $N$ is a synthetic object $(\psi,\xi,\eta,\gamma)$ consisting of a $(1,1)$-tensor field $\psi$, a vector field $\xi\in\mathcal{X}(N)$, a $1$-form $\eta$, and a Riemannian metric $\gamma$, satisfying the following identities:
$$\psi^{2}=-I+\eta\otimes\xi$$
$$\eta\circ\psi=0,\ \eta(\xi)=1$$
$$\gamma(\psi X,\psi Y)=\gamma(X,Y)-\eta(X)\eta(Y)$$
$$[\psi,\psi]+2(d\eta)\otimes\xi=0$$
$$d\eta=c\phi,$$
where the $2$-form $\phi$ is given by $\phi(X,Y)=\gamma(X,\psi Y)$. A $(2n-1)$-dimensional manifold $N$ carrying a $c$-Sasakian structure is a $c$-Sasakian manifold. Of course, one may always go back to a usual Sasakian structure by a transformation:
$$\hat{\psi}=\psi,\ \hat{\xi}=\frac{1}{c}\xi,\ \hat{\eta}=c\eta,\ \hat{\gamma}=c^{2}\gamma.$$
We then have
\begin{proposition}(\cite{Vaisman1979} and \cite[Proposition 5.1]{DO})
	Let $M$ be a Vaisman manifold. Let $S$ be a leaf of $\mathcal{F}_{0}$ and $i: S\hookrightarrow M$ the inclusion. Let $(\psi,\xi,\eta,\gamma)$ on $S$ be given by
	$$\psi=J\circ(di)+(i^{\ast}v)\otimes(U\circ i),\ \xi=V\circ i,\ \eta=i^{\ast}v,\ \gamma=i^{\ast}g,$$
	Then $(\psi,\xi,\eta,\gamma)$ is a $c$-Sasakian structure on $S$.
\end{proposition}
Using $$g=\gamma+u\otimes u$$ and the De Rham decomposition theorem, we obtain
\begin{proposition}(\cite{Vaisman1979} and \cite[Proposition 5.2]{DO})\label{P1} 
The universal Riemannian covering manifold  $M$ of a complete Vaisman manifold is the Riemannian product of a simply connected $c$-Sasakian manifold $N$, which is the universal covering space of a leaf $N$ of $\mathcal{F}_{0}$ and the real line.
\end{proposition}
\begin{proof}[\textbf{Proof of Theorem \ref{T2}}.]
The conclusion follows from Theorem \ref{T3}.
\end{proof}
\begin{remark}
Following Proposition \ref{P1}, it implies that every complete simply-connected Vaisman manifold is the Riemannnian product of the real line $\mathbb{R}$ and a complete simply-connected manifold $N$. A well-known fact is that the K\"{u}nneth formula holds for spaces of $L^{2}$-harmonic forms \cite{Zucker}. Since the spaces of $L^{2}$-harmonic forms on $\mathbb{R}$ is trivial, we get $\bar{H}^{k}_{(2)}(N\times\mathbb{R})=0$ for all $k$, i.e., there is no non-trivial $L^{2}$-harmonic forms on complete simply-connected Vaisman manifolds. Here $\bar{H}^{\ast}_{(2)}$ denotes reduced $L^{2}$-cohomology. 
\end{remark}
\section*{Acknowledgements}
We would like to thank Professor H.Y. Wang for drawing our attention to the Vaisman manifold and generously helpful suggestions about these. We would also like to thank the anonymous referee for careful reading of my manuscript and helpful comments. This work is supported by Natural Science Foundation of China No. 11801539 (Huang), No. 11701226 (Tan) and Natural Science Foundation of Jiangsu Province BK20170519 (Tan).

\end{document}